\documentclass[12pt]{amsart}
\usepackage{amsmath, amssymb, color}
\usepackage{amsthm}
\usepackage{color}
\usepackage[dvipdfm]{graphicx}
\usepackage{multirow,bigdelim}
\usepackage{graphicx}
\usepackage{array}
\usepackage{ifpdf}
\usepackage{epsfig}
\usepackage{amscd}
\usepackage{graphics}
\numberwithin{equation}{section}
\theoremstyle{definition}
\newtheorem{thm}{Theorem}[section]

\newtheorem{lem}[thm]{Lemma}

\newtheorem{defn}[thm]{Definition}

\newtheorem{rem}[thm]{Remark}

\newcommand{\ve}{\varepsilon}

\title{Branched twist spins and knot determinants}
\author{Mizuki Fukuda}
\address{Graduate School of Science, Tohoku University, Sendai, 980-8578, Japan}
\email{sb4m24@math.tohoku.ac.jp}

\begin{document}
\maketitle
\vspace{-5mm}
\begin{abstract}
A branched twist spin is a generalization of twist spun knots, which appeared in the study of locally smooth circle actions on the $4$-sphere due to Montgomery, Yang, Fintushel and Pao.
In this paper, we give a sufficient condition to distinguish non-equivalent, non-trivial branched twist spins by using knot determinants.
To prove the assertion, we give a presentation of the fundamental group of the complement of a branched twist spin, which generalizes a presentation of Plotnick, calculate the first elementary ideals and obtain the condition of the knot determinants by substituting $-1$ for the indeterminate.
 
\end{abstract}
\section{Introduction}
An $n$-knot is an $n$-sphere embedded into the $(n+2)$-sphere.
In this paper, we assume that all knots are smooth.  
A non-trivial example of a $2$-knot was first introduced by Artin~\cite{A},
which was constructed from a $1$-knot in $S^3$ by rotating it in $S^4$ along a trivial axis.
It is called a {\it spun knot}.
One of the important properties of spun knots is 
that its knot group is isomorphic to the knot group of the $1$-knot of rotation. 
Afterwards, Zeeman generalized Artin's construction by adding a twist during the rotation,
called a {\it twist spun knot}~\cite{Z}.
Such a twisting is realized by an $S^1$-action on $S^4$ having two fixed points 
and one family of exceptional orbits. A twist spun knot is then obtained as
the union of the fixed points and the preimage of a (possibly knotted) arc in the orbit space connecting the image of 
the two fixed points.

Suppose that $S^4$ has an effective locally smooth $S^1$-action.
Let $E_m$ be the set of exceptional orbits of $\mathbb{Z}_m$-type, where $m$ is an integer greater than $1$, and $F$ be the fixed point set. 
Let $E_m^{\ast}$ and $F^{\ast}$ denote the image of the orbit map of $E_m$ and $F$, respectively.
In the study of $S^1$-action on $S^4$, 
Montgomery and Yang~\cite{MY} showed that effective locally smooth $S^1$-actions are classified into the following four types: 
(1) $\{D^3\}$, (2) $\{S^3\}$, (3) $\{S^3,m\}$, and (4) $\{(S^3,K),m,n\}$, which are called {\it orbit data}.
The $3$-ball $D^3$ and the $3$-sphere $S^3$ in these notations represent the orbit spaces.
Type (1) has no exceptional orbit and $F^{\ast}$ is the whole boundary of $D^3$. 
Type (2) has no exceptional orbit and $F^{\ast}$ consists of two points. 
Type (3) has one type of exceptional orbits $E_m$ 
whose image $E_m^{\ast}$ constitutes an arc in the orbit space $S^3$.
 The image $F^{\ast}$ is the endpoints of the arc $E_m^{\ast}$. 
Type (4) has two different types of exceptional orbits $E_m$ and $E_n$, where $m, n$ are larger than $1$ and relatively prime.
The images $E_m^{\ast}$ and $E_n^{\ast}$ are arcs with the endpoints $F^{\ast}$ such that 
$E_m^{\ast} \cup E_n^{\ast}\cup F^{\ast}$ constitutes a $1$-knot in $S^3$.
Conversely, Fintushel~\cite{Fi} and Pao~\cite{Pa} showed that there is a unique weak equivalence class of effective locally smooth $S^1$-actions on $S^4$ for each orbit data.
Here the weak equivalence of $S^1$-manifolds $M_1$ and $M_2$ is defined by a homeomorphism 
$H : M_1\to M_2$ satisfying $H(\theta x)=a(\theta)H(x)$ 
for $\theta \in S^1$ and $x\in M_1$, where $a$ is an automorphism of $S^1$.

The sets $E_m \cup F$ and $E_n \cup F$ in type (4) are diffeomorphic to 2-spheres and the image of all exceptional orbits and the fixed point set is a 1-knot $K = E^{\ast}_m \cup E^{\ast}_n \cup F^{\ast}$ in $S^3$. 
An $(m,n)$-$branched\ twist\ spin$ of $K$ is defined to be the $2$-knot $E_n \cup F$ in the case of (4), which corresponds to the case $m,n>1$, and defined similarly in type (2) and (3).
See Section $1$  for precise definition.
In this paper, by fixing the orientation of $S^4$, we generalize
the definition of $(m,n)$-branched twist spins for
$(m,n)\in\mathbb Z\times \mathbb N$
, where $|m|$ and $n$ are copirme.
Note that if $n=1$ then it is of type (3) and corresponds to
the $m$-twist spun knot, 
and if $m=0$ and $n=1$ then it is of type (2) and corresponds to the spun knot.

An $n$-knot $K$ is said to be equivalent to another $n$-knot $K^{\prime}$,
denoted by $K \sim K^{\prime}$, if there exists a smooth isotopy $H_t: S^{n+2}\to S^{n+2}$ such that $H_0={\rm id}$ and $H_1(K)=K^{\prime}$.

It is known by Hillman and Plotnick that if an $1$-knot $K$ is torus or hyperbolic then $K^{m,n}$ is not reflexive for $m\geq n$ and $m \geq 3$~\cite{HP}.
Here a $2$-knot is said to be reflexive if its Gluck reconstruction is equivalent to either its mirror image or its orientation reverse. 
Since the trivial $2$-knot is reflexive, $K^{m,n}$ is non-trivial if $K$ is torus or hyperbolic and $m\geq n$ and $m \geq 3$.
Though their result detects non-triviality of these branched twist spins, in the best of my knowledge, there is no result which distinguishes non-equivalent, non-trivial branched twist spins.

In this paper we give a sufficient condition to distinguish non-trivial branched twist spins by using knot determinants.
Our main theorem is the following:

\begin{thm}\label{thm} 
Let $K_1^{m_1 , n_1} ,  K_2^{m_2 , n_2}$ be branched twist spins constructed from $1$-knots $K_1$ and $K_2$ in $S^3$, respectively.
\begin{itemize}
\item[(1)]  
If $m_1 , m_2$ are even and $\left| \Delta_{K_1}(-1) \right| \neq \left| \Delta_{K_2}(-1) \right|$ then $K_1^{m_1 , n_1}\sim \hspace{-11pt}/\ \ K_2^{m_2 , n_2}$.
\item[(2)] 
If $m_1$ is even, $m_2$ is odd and $|\Delta_{K_1}(-1)| \neq 1$ then 
$K_1^{m_1 , n_1}\sim \hspace{-11pt}/\ \ K_2^{m_2 , n_2}$.
\end{itemize}
\end{thm} 

To prove this theorem, we first give a presentation of the fundamental group of the complement of a branched twist spin, which generalizes a presentation of Plotnick, and obtain the first elementary ideal from this presentation by using Fox calculus.
Then we obtain certain equations of elementary ideals and the knot determinants appear when we substitute $-1$ for the variable $t$ of the Alexander polynomials.

This paper is organized as follows:
In Section $2$, we give the definition of a branched twist spin and the presentation of its fundamental group. 
In Section $3$, we give concrete generators of the first elementary ideal and prove Theorem~\ref{thm}.

\section{Branched twist spins}

For a point $x\in M$, an orbit $G(x)$ is called $G/H$-type if the isotropy group $G_x=\{gx=x\ |\ g\in G\}$ of $x$ is $H\subset G$.
If the $G$-action is locally smooth, each orbit of $G/H$-type has a linear slice $S$ such that $H$ acts orthogonally on $S$.
See~\cite{B} for the terminologies of transformation groups.

Suppose that $S^4$ has an effective locally smooth $S^1$-action.
We consider the $S^1$-actions of type (3) and type (4) in Section $1$.
Let $E_m$ and $E_n$ be the sets of exceptional orbits of $\mathbb{Z}_m$-type and $\mathbb{Z}_n$-type, respectively, and $F$ be the fixed point set.
Here, when $n=1$, the $S^1$-action is of type (3).
Let $E_m^{\ast}$, $E_n^{\ast}$ and $F^{\ast}$ denote the image of the orbit map of $E_m$, $E_n$ and $F$, respectively. 
Then $E_m \cup F$ and $E_n \cup F$ are diffeomorphic to 2-spheres and the image of all exceptional orbits and the fixed point set constitutes an $1$-knot $K = E^{\ast}_m \cup E^{\ast}_n \cup F^{\ast}$.

Now we give the definition of $(m,n)$-branched twist spins for $(m,n)\in\mathbb Z\times\mathbb N$.
To do this, we need to fix orientations of $S^4$ and the $S^1$-action
on $S^4$, and observe the direction of twisting in neighborhoods
of the exceptional orbits.
Let $(m,n)$ be a pair of integers in $\mathbb Z\times\mathbb N$
such that $|m|$ and $n$ are coprime. Here we further assume that $m\ne 0$.
First, we decompose the orbit space $S^3$ into five pieces.
The set $F^{\ast}$ consists of two points, say $x^{\ast}_1$ and $x^{\ast}_2$, 
and let $D^{3\ast}_i$ be a small compact ball in $S^3$ centered at $x^{\ast}_i$ for $i=1,2$. 
Choose a compact tubular neighborhood $N(K)$ of $K$ sufficiently small such that 
$N(K) \setminus \text{int}(D^{3\ast}_1 \cup D^{3\ast}_2)$ has two connected components $N_m$ and $N_n$ 
with $N_m\cap E_m^{\ast} \neq \emptyset$ and $N_n\cap E_n^{\ast} \neq \emptyset$.
Set $E^{c\ast}_m$ and $E^{c\ast}_n$ to be the connected components of $K \setminus \text{int}(D^{3\ast}_1 \cup D^{3\ast}_2)$, 
where $E_m^{c\ast}\subset N_m$ and $E_n^{c\ast}\subset N_n$, respectively.
Note that $N_m$ and $N_n$ are diffeomorphic to
$E^{c\ast}_m\times D^2$ and $E^{c\ast}_n\times D^2$, respectively.
Let $X$ be the closure of $ S^3 \setminus ((E^{c\ast}_m \times D^2) \cup (E^{c\ast}_n \times D^2) \cup D^{3\ast}_1\cup D^{3\ast}_2)$, which is the knot complement of $K$.
Then we have a decomposition $S^3 = X \cup (E^{c\ast}_m \times D^2) \cup (E^{c\ast}_n \times D^2) \cup D^{3\ast}_1\cup D^{3\ast}_2$.

\begin{figure}[htbp]
\centering
\includegraphics[scale=0.4]{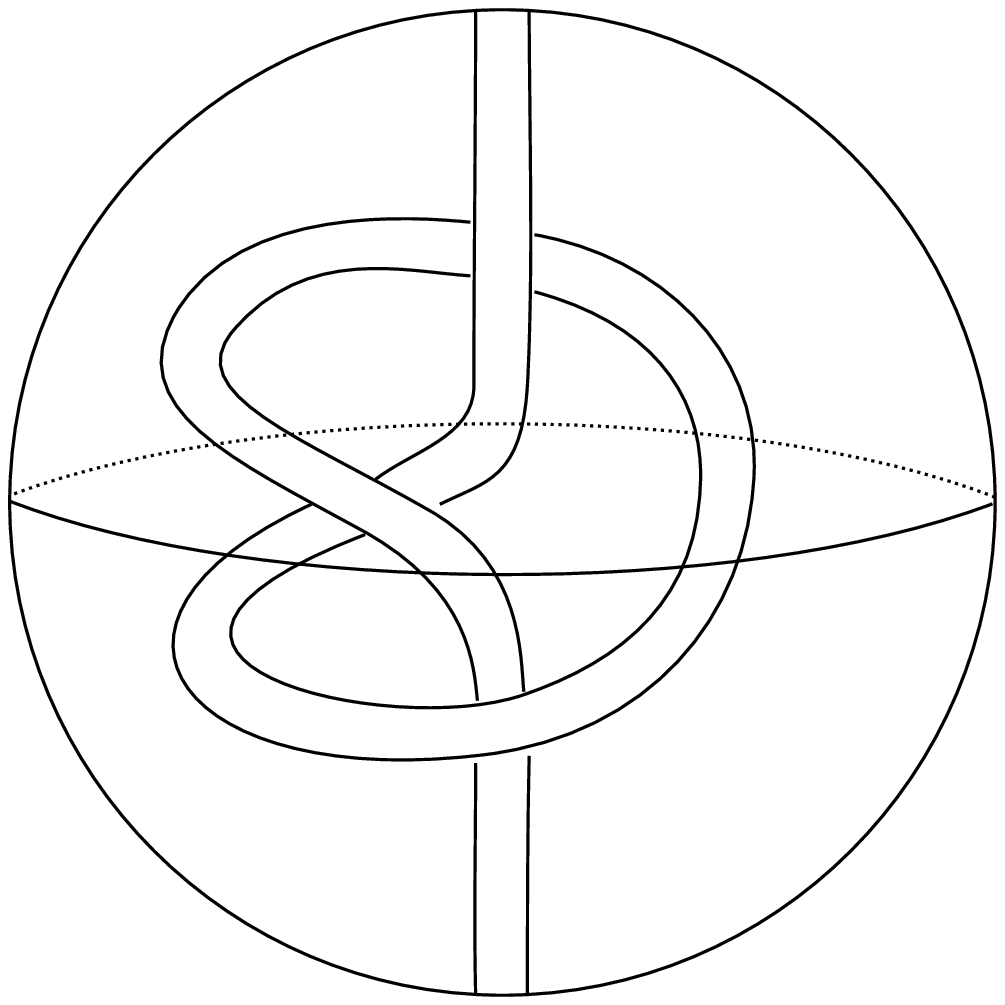}
\caption{The complement $X$}\label{comp}
\end{figure}

Let $p : S^4\to S^3$ be the orbit map.
Each point of $X$ is the image of a free orbit.
Thus $p|X\times S^1$ is a principal $S^1$-bundle.
The preimage $p^{-1}(X)$ is diffeomorphic to $X\times S^1$ and $p|X\times S^1: X\times S^1\to X$ is the first projection 
since $H^2(X ; \mathbb{Z}) = H_1(X , \partial X ; \mathbb{Z}) = 0 $ (cf.~\cite[Chapter 2]{Bl}).

Let $B^4_i$ be a linear slice at $p^{-1}(x^{\ast}_i)$, which is a closed 4-ball.
By~\cite{Fi}, $p^{-1}(D^{3\ast}_i)$ is diffeomorphic to $B^4_i$ and the action on $p^{-1}(D^{3\ast}_i)$ is $S^1$-equivalent to that on $B^4_i$.
Note that the action on $B^4_i$ is the cone of the action of $\partial B^4_i$ and so the action on $p^{-1}(D^{3\ast}_i)$ is the cone of the action of $p^{-1}(\partial D^{3\ast}_i)$.

Choosing a point $z^{\ast}_m$ in $E^{c\ast}_m$, let $D^{2\ast}_{z^{\ast}_m}$ be a 2-disk in $S^3$ centered at $z^{\ast}_m \in E^{c\ast}_m$ and transversal to $E^{c\ast}_m$.
The preimage $p^{-1}(D^{2\ast}_{z^{\ast}_m})$ is a solid torus $V_m$ whose core is the exceptional orbit of $\mathbb{Z}_m$-type.
In the same way, choosing a point $z^{\ast}_n \in E^{c\ast}_n$, 
let $D^{2\ast}_{z^{\ast}_n}$ be a 2-disk in $S^3$ centered at $z^{\ast}_n \in E^{c\ast}_n$ and transversal to $E^{c\ast}_n$.
The preimage $p^{-1}(D^{2\ast}_{z^{\ast}_n})$ is a solid torus $V_n$ whose core is the exceptional orbit of $\mathbb{Z}_n$-type.
Note that, since $V_m\cup V_n=p^{-1}(\partial D^{3\ast}_i)$ is a $3$-sphere, $p^{-1}(y^{\ast})\ (y^{\ast}\in \partial D^{2\ast}_{z_m})$ is a curve on $\partial V_m$ rotating, up to orientation, $m$ times along the meridian and $n$ times along the longitude of $V_m$,
where $n$ is determined module $m$ due to the self-homeomorphisms of $V_m$.

Now we fix the orientations of $V_m$ and $E^{c\ast}_m$ as follows:
First, fix the orientation of $S^4$ and those of orbits such that they coincide with the direction of the $S^1$-action.
These orientations determine the orientation of $V_m\times E_m^{c*}$.
Let  $(\phi,\theta)$ be a preferred meridian-longitude pair of $X$.
From the decomposition of the orbit space $S^3$,
we can see that $\phi$ is regarded as a coordinate of the second factor of $V_m\times E^{c\ast}_m$.
We assign the orientation of $V_m$ so that the orientation of $V_m\times E^{c\ast}_m$ coincides with the given orientation of $S^4$.
Finally, we choose the meridian and longitude pair $(\Theta,H)$ 
of $V_m\cong D^2\times S^1$ such that $H$ becomes the meridian
of $V_n$ in the decomposition $V_m\cup V_n=p^{-1}(\partial D_i^{3*})$
and the orbits of the $S^1$-action are in the direction 
$\ve n\Theta+|m|H$ with $n>0$, where $\ve= 1$ if $m\geq 0$ and $\ve = -1$ if $m<0$. 

\begin{figure}[htbp]
\centering
\includegraphics[scale=0.6]{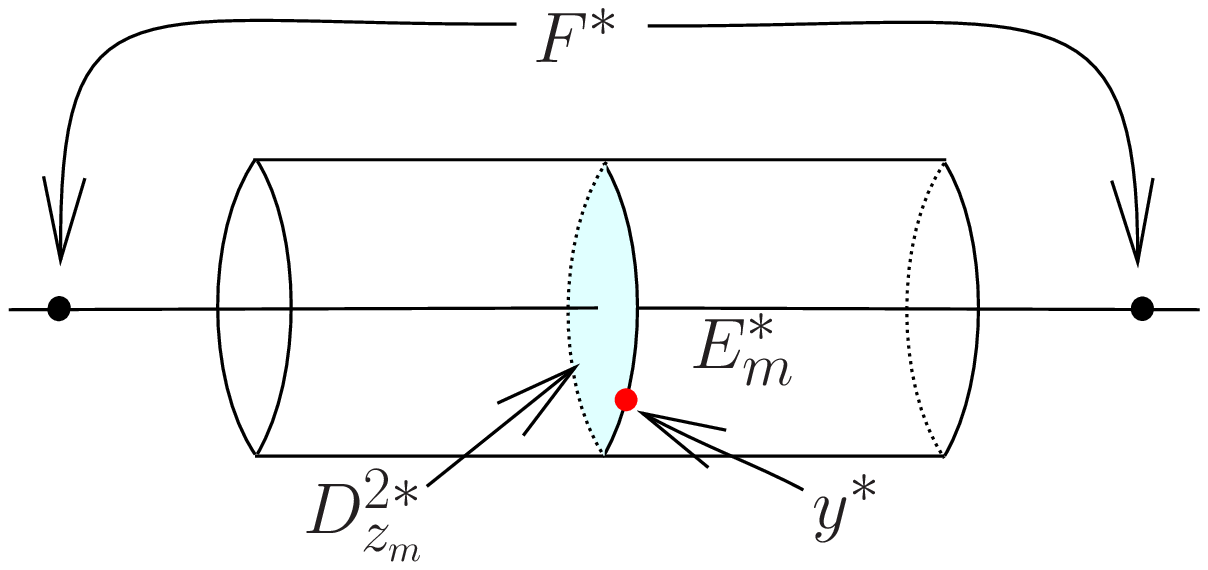}
\ \ \ \ \ \ \ 
\includegraphics[scale=0.6]{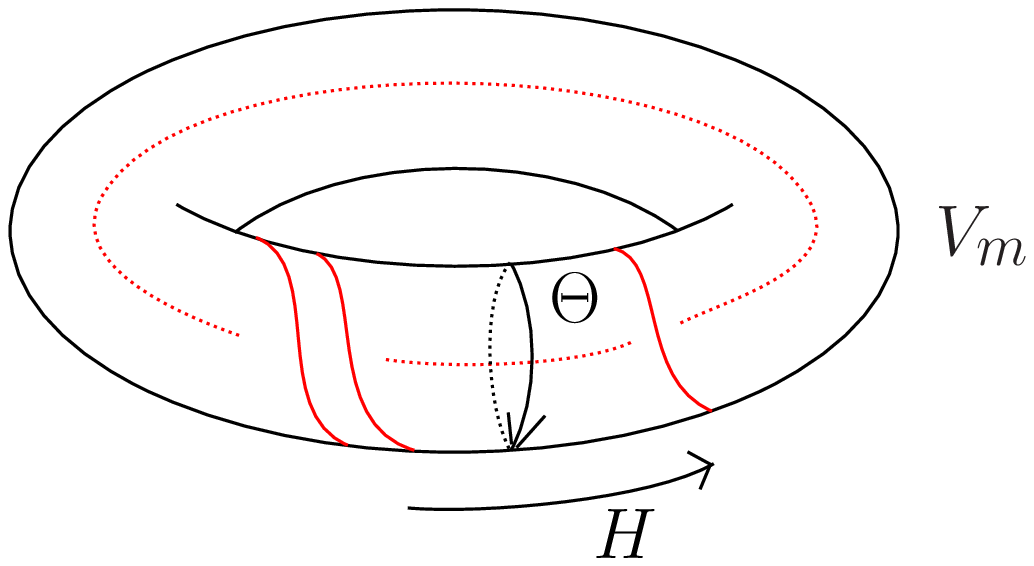}
\caption{The preimage of $y^{\ast}$}\label{exp}
\end{figure}

\begin{defn}[Branched twist spin]
For each pair $(m,n)\in\mathbb{Z}\times \mathbb{N}$ with 
$m\ne 0$ such that $|m|$ and $n$ are coprime,
let $K^{m,n}$ be the $2$-knot $E_n\cup F$.
If $(m,n)=(0,1)$ then let $K^{0,1}$ be the spun knot of $K$.
The $2$-knot $K^{m,n}$ is called an $(m,n)$-$branched\ twist\ spin$ of $K$.
\end{defn}

Note that the branched twist spin $K^{m,1}$ constructed from $\{(S^3 , K),m,1\} $ is an $m$-twist spun knot of $K$.

\begin{figure}[htbp]
\centering
\includegraphics[scale=0.5]{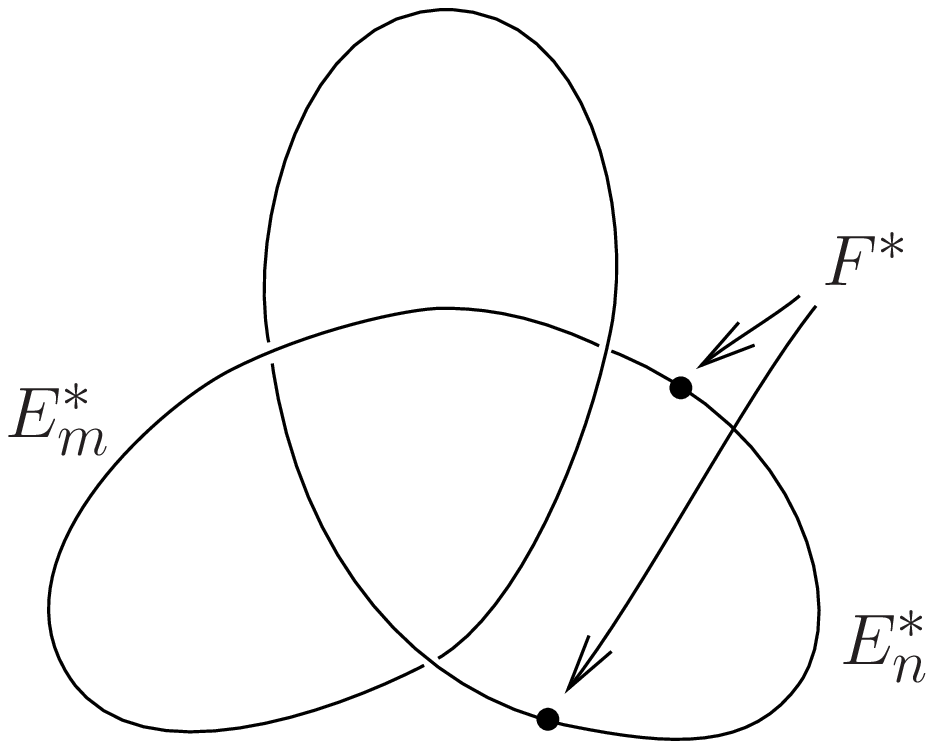}
\caption{The image of $E_m^{\ast}\cup E_n^{\ast} \cup F^{\ast}$ in $S^3$}\label{test}
\end{figure}

Let $K$ be an $n$-knot in $S^{n+2}$. 
The fundamental group of the knot complement $S^{n+2}\setminus \text{int}N(K)$ of an $n$-knot $K$ is called the {\it knot group} of $K$, 
where $N(K)$ is a tubular neighborhood of $K$.
\begin{lem}\label{P}
Let $K$ be an $1$-knot and $K^{m,n}$ be the $(m,n)$-branched twist spin of $K$ with $(m,n)\in\mathbb Z\times \mathbb N$, where $|m|$ and $n$ are coprime.
Let $\langle x_1,\ldots,x_s\mid r_1,\ldots,r_t\rangle$
be a presentation of the knot group of $K$ such that $x_1$ is a meridian.
Then the knot group of $K^{m,n}$ has the presentation
\begin{equation}\label{eq1}
\pi_1(S^4 \setminus \text{int}N(K^{m , n})) \cong 
\langle x_1,\ldots,x_s,h\ |\ r_1,\ldots,r_t,x_ihx^{-1}_ih^{-1}, x_1^{|m|} h^{\beta} \rangle,
\end{equation} 
where $\beta$ is an integer such that
$n\beta\equiv \ve$ (mod $m$).
Recall that $\ve = 1$ if $m\geq 0$ and $\ve = -1$ if $m<0$.
\end{lem}

\begin{proof}

The knot complement of $K^{m,n}$ is given by $X\times S^1 \cup_f V_m\times E_m^{c\ast}$, where 
the map $f:\partial D^{2\ast}\times E_m^{c\ast}\times S^1 \to \partial V_m\times E_m^{c\ast}$ is the attaching map specified by the decomposition explained in this section.

Let $h$ is the coordinate of the second factor of  $X\times S^1$ whose direction coincides
with the $S^1$-action.
By the above discussion of the orientations, 
the induced map $f_{\ast}:H_1(\partial D^{2\ast}\times E_m^{c\ast}\times S^1) \to H_1(V_m\times E_m^{c\ast})$ must satisfy

$$
(f_{\ast}([\theta]) , f_{\ast}([h])) = 
([\Theta] , [H])
\left(
\begin{array}{cc}
\alpha & \ve n\\
-\beta &  |m|
\end{array}
\right),
$$
where $\alpha$ and $\beta$ are integers satisfying $m\alpha + n\beta = \ve$.

The inverse of $f_{\ast}$ satisfies the relation
$$
(f^{-1}_{\ast}([\Theta]) , f^{-1}_{\ast}([H])) = 
([\theta] , [h])
\left(
\begin{array}{cc}
|m| & -\ve n\\
\beta & \alpha
\end{array}
\right),
$$
hence $f_{\ast}^{-1}([\Theta]) = |m|[\theta] + \beta[h]$ and $f_{\ast}^{-1}([H]) =  -\ve n[\theta] + \alpha[h]$ hold. 
Since $f_{\ast}^{-1}([\Theta])$ is null-homologous in $V_m \times E^{c\ast}_m$, 
$$
\pi_1(S^4 \setminus \text{int}N(K^{m , n})) \cong 
\langle x_1,\ldots,x_s,h\ |\ r_1,\ldots,r_t,x_ihx^{-1}_ih^{-1},x_1^{|m|} h^{\beta} \rangle
$$
holds by Van Kampen's theorem. 
\end{proof}

\begin{rem}\label{meridian}
Since $\partial D_i^4$ is the union of  $V_m$ and $V_n$, the meridian $\mu$ of $K^{m , n}$ is that of $V_n$, named $H$ in the above proof. 
From the relation $f_{\ast}^{-1}([H]) = -\ve n[\theta] + \alpha[h]$, we have $\mu = \theta^{-\ve n}h^{\alpha}$.
\end{rem}

\section{Proof of main theorem}
Let $K$ be an $n$-knot. 
Assume that a presentation $ \langle x_1, \ldots , x_l\ |\ r_1, \ldots , r_k\rangle$ of $\pi_1(S^{n+2}\setminus \text{int}N(K))$ is given.
Let $a:\pi_1(X)\to H_1(X)\cong \mathbb Z\langle t\rangle$
be the quotient map. This map induces a map
$a_{\ast}:\mathbb Z\pi_1(X)\to \mathbb Z[t,t^{-1}]$ naturally.
The matrix $A$ defined by 
$$
A=\left( a_{\ast} \left(\cfrac{\partial r_i}{\partial x_j}   \right)   \right)
$$
is called the $Alexander\ matrix$ of $K$.
Note that $H_1(X) \cong \mathbb{Z}$ holds for all $n$-knots, and $a$ takes a meridian of $K$ to the generator $t$ of $\mathbb{Z}$.

In the case of a $1$-knot, it is very common to use
a Wirtinger presentation for describing the knot group of $K$.
Then, the quotient map of the abelianization
maps each generator to the generator $t$ of $H_1(X;\mathbb Z)$.

Two Alexander matrices $A$ and $A^{\prime}$ is said to be equivalent, denoted by $A\sim A^{\prime}$, 
if $A^{\prime}$ is obtained from $A$ by the following operations:
(1)Permuting rows or permuting columns. 
(2) Adjoining to a row or a column a linear combination of other rows or columns, respectively.
(3)
$
A\to 
\left(
\begin{array}{ccc}
\text{$A$}\\
0
\end{array}
\right).
$
(4)  
$
A\to 
\left(
\begin{array}{ccc}
\text{$A$}&0\\
0&1
\end{array}
\right).
$

For the Alexander matrix $A\in M(p,q,\mathbb Z[t,t^{-1}])$ of $K$ over $\mathbb Z[t,t^{-1}]$ and non-negative integer $k$, 
the $k$-th  $elementary\ ideal$ $E_k(A)$ of $K$ is defined as follows: 
\begin{itemize}
 \item $E_k(A)$ is the ideal generated by determinants of all $(q-k) \times (q-k)$-submatrices of  $A$ if $0 < q-k \leq p$,
 \item $E_k(A)=0$  if $q-k > p$,
 \item $E_k(A)=\mathbb{Z}[t,t^{-1}]$ if $q-k \leq 0$.
\end{itemize}

For all $k$, we have $E_k \subset E_{k+1}$.
If $K\sim K^{\prime}$, a presentation of the knot group of $K^{\prime}$ is obtained from that of $K$ by Tietze transformations. 
Since Tietze transformations preserve the equivalence class of Alexander matrices, we denote $E_k(A)$ by $E_k(K)$.

Now we study the elementary ideals of the knot group of $K^{m,n}$.
Hereafter we fix a Wirtinger presentation of the knot group $K$ as
\begin{equation}\label{Wir}
\langle x_1,\ldots,x_l\ |\ r_1,\ldots,r_l\rangle.
\end{equation}
Then, 
\begin{equation}\label{eq2}
\pi_1(S^4 \setminus \text{int}N(K^{m , n})) \cong 
\langle x_1 , \ldots , x_l , h\ |\ r_1 , \ldots , r_l , x_ihx_i^{-1}h^{-1}, x_1^{|m|}h^{\beta} \rangle
\end{equation}
holds by applying Lemma~\ref{P}.
Let $r_{l+i}$ be $x_ihx_i^{-1}h^{-1}$ for each $i(1\leq i \leq l)$, and let $r_{2l+1}$ be $x_1^{|m|}h^{\beta}$.
Then, using the induced map $a_{\ast}$ and Fox calculus for this presentation, 
the Alexander matrix $A$ of $K^{m,n}$ is written as

$$A=
\left(
a_{\ast}\left(\frac{\partial r_i}{\partial x_j}\right)
\right)
.
$$

\begin{lem}\label{Pl}
 The $k$-th elementary ideal of $K^{m,n}$ has the following property:
\begin{itemize}
\item[(1)]
$E_0(K^{m , n}) = 0$.
\item[(2)] 
Let $\beta$ be a positive integer satisfying $n\beta \equiv \ve\ (\text{mod}\ m)$.
The ideal $E_1(K^{m , n})$ is the ideal generated by the following elements:
\begin{equation*}
\left\{
\begin{split}
&\Delta_{K }(t^{\beta})\left\{
(1-t^{|m|}) ,
(1-t^{\beta}) , 
\frac{1-t^{|m|\beta}}{1-t^{\beta}} , 
\frac{1-t^{|m|\beta}}{1-t^{|m|}}
\right\} , \\
&G_i(t^{\beta})(1-t^{|m|})\left\{
(1-t^{|m|}) , 
(1-t^{\beta}) , 
 \frac{1-t^{|m|\beta}}{1-t^{\beta}} ,
 \frac{1-t^{|m|\beta}}{1-t^{|m|}} 
 \right\} , \\
&(1 - t^{|m|})^{l-1}\left\{
(1-t^{|m|}) , 
(1-t^{\beta}) , 
 \frac{1-t^{|m|\beta}}{1-t^{\beta}} ,
 \frac{1-t^{|m|\beta}}{1-t^{|m|}} 
 \right\},
\end{split}
\right.
\end{equation*}
where $l$ is the number of generators of the knot group of $K$ and $G_i(t)$ are generators of $E_2(K)$.
Especially, $E_1(K^{m , n}) \neq 0$.
\end{itemize}
Here the notation $P\{Q_1,Q_2,Q_3,Q_4\}$ means $PQ_1,PQ_2,PQ_3,PQ_4$.
\end{lem}

\begin{proof}
From Lemma~\ref{P}, we have the presentation~\eqref{eq2}
of the knot group of $K^{m,n}$.
By Remark~\ref{meridian}, the meridian of $K^{m,n}$ is written as $x^{-\ve n}_1 h^{\alpha}$, where $m\alpha +n\beta = \ve $.
Since the quationt map $a$ sends $x^{-\ve n}_1 h^{\alpha}$ to the generator $t$ of  $H_1(X ; \mathbb{Z})$,
we have $a(x_1) = \cdots = a(x_l) =t^{-\beta} , a(h) = t^{|m|}$. 
Then the Alexander matrix $A$ obtained from~\eqref{eq2} is  given by
\begin{equation}\label{eq3}
\scalebox{0.7}{$\displaystyle
\text{\Huge{$A\ =$}} 
\ \ \ 
\left(
\begin{array}{cccccccc}
&&&&0&\\
&&&&\vdots&\\
&\multicolumn{2}{c}{$\mbox{\smash{\Huge $B$}}$} &&\vdots\\
&&&&\vdots&\\
&&&&0&\\ 
1-t^{|m|}&&&&t^{-\beta}-1&\\ 
&\ddots&&\text{\Huge{$O$}}&\vdots\\
\text{\Huge{$O$}}&&\ddots&&\vdots\\
&&&1-t^{|m|}&t^{-\beta}-1\\
\cfrac{1-t^{-|m|\beta}}{1-t^{-\beta}}&0&\cdots&0&\cfrac{t^{-|m|\beta}(1-t^{|m|\beta})}{1-t^{|m|}}
\end{array}
\right),
$}
\end{equation}
where 
$B$ is the Alexander matrix of $K$ obtained from the Wirtinger presentation~\eqref{Wir} by replaced $t$ with $t^{-\beta}$.
Since $r_i$ is $x_i x_j x^{-1}_i x^{-1}_k$, for each $i\ (1\leq i \leq l)$, entries in the $i$-th row of $A$ satisfy
$$
a_{\ast}\left(\frac{\partial x_i x_j x^{-1}_i x^{-1}_k}{\partial x_p}\right) =
\left\{
\begin{array}{cc}
1- t^{-\beta}&\ \ \  (p = i)\\
t^{-\beta}&\ \ \  (p = j)\\
-1&\ \ \  (p = k)\\
0&\ \ \ \  (p: \text{others})
\end{array} 
\right.
.
$$
Therefore we have
$\sum^l_{p = 1} a_{\ast}\left(\frac{\partial x_i x_j x^{-1}_i x^{-1}_k}{\partial x_p}\right)= 0.$
Hence, $A$ is equivalent to 
\begin{eqnarray*}
\scalebox{0.7}{$\displaystyle
\text{\Huge{$A^{\prime}\ =$}} 
\ \ \ 
\left(
\begin{array}{cccccr}
0&&&&0&\\
\vdots&&&&\vdots&\\
\vdots&\multicolumn{3}{c}{$\mbox{\smash{\Huge $\hspace{10pt} B_1$}}$} &\vdots\\
\vdots&&&&\vdots&\\
0&&&&0& \\ 
1-t^{|m|}&&&&t^{-\beta}-1&\\
&\ddots&&\text{\Huge{$O$}}&0&\\
&&\ddots&&\vdots\\
\text{\Huge{$O$}}&&&&\vdots\\
&&&1-t^{|m|}&0\\
\cfrac{1-t^{-|m|\beta}}{1-t^{-\beta}}&0&\cdots&0&\cfrac{t^{-|m|\beta}(1-t^{|m|\beta})}{1-t^{|m|}}
\end{array}
\right),
$}
\end{eqnarray*}
where $B_1$ is an $l \times (l-1)$ matrix.
 An $(l + 1) \times (l + 1)$ submatrix of $A^{\prime}$ should contain both the $(l + 1)$-th row and the $(2l + 1)$-th row
 if it's determinant is not zero. Then any $(l+1)\times (l+1)$ submatrix has the form
$$
\left(
\begin{array}{cccccr}
0&&&&0\\
\vdots&&\multicolumn{1}{c}{$\mbox{\smash{\Huge $\hspace{10pt} B_2$}}$} &&\vdots\\
0&&&&0\\
1-t^{|m|}& 0&\cdots&0&t^{-\beta}-1\\
0&&&&0\\
\vdots&&\multicolumn{1}{c}{$\mbox{\smash{\Huge $\hspace{10pt} B_3$}}$} &&\vdots\\
0&&&&0\\
\frac{1-t^{-|m|\beta}}{1-t^{-\beta}}&0&\cdots&0&\frac{t^{-|m|\beta}(1-t^{|m|\beta})}{1-t^{|m|}}
\end{array}
\right),
$$
where 
$
\left(
\begin{array}{ccc}
\text{$B_2$}\\
\text{$B_3$}
\end{array}
\right)
$
 is an $(l-1)\times(l-1)$ matrix, and its determinant has the factor
$$
\left|
\begin{array}{cc}
1-t^{|m|}&t^{-\beta}-1\\
\frac{1-t^{-|m|\beta}}{1-t^{-\beta}}&\frac{t^{-|m|\beta}(1-t^{|m|\beta})}{1-t^{|m|}}
\end{array}
\right| 
$$
which is equal to zero.
Thus $E_0(K^{m , n}) = 0$ holds.

By checking all the determinants of $l \times l$ submatrices of $A^{\prime}$,
we can see that $E_1(K^{m,n})$ is generated by the following terms:

\begin{equation*}
\left\{
\begin{split}
&\Delta_{K }(t^{-\beta})\left\{
(1-t^{|m|}) , 
(t^{-\beta}-1) ,
\frac{1-t^{-|m|\beta}}{1-t^{-\beta}} , 
\frac{t^{-|m|\beta}(1-t^{|m|\beta})}{1-t^{|m|}}
\right\}  ,\\
&G_i(t^{-\beta})(1-t^{|m|})\left\{
(1-t^{|m|}) ,
(t^{-\beta}-1) ,
\frac{1-t^{-|m|\beta}}{1-t^{-\beta}} ,
\frac{t^{-|m|\beta}(1-t^{|m|\beta})}{1-t^{|m|}}
\right\}  ,\\
&(1 - t^{|m|})^{l-1}\left\{
(1-t^{|m|}) ,
(t^{-\beta}-1) ,
\frac{1-t^{-|m|\beta}}{1-t^{-\beta}} ,
\frac{t^{-|m|\beta}(1-t^{|m|\beta})}{1-t^{|m|}}
\right\} .
\end{split}
\right.
\end{equation*}
Here 
$\Delta_K(t)$ is the Alexander polynomial of $K$, which is given, up to unit, as the common factor of all determinants of $(l-1)\times(l-1)$ submatrices of $B_1$,
 and $G_i(t)$ are the generators of $E_2(K)$.
Thus we have the assertion.
\end{proof}

\begin{proof}[Proof of Theorem~\ref{thm}]
We prove the assertion by contraposition. 
Suppose that $K_1^{m_1 , n_1} \sim K_2^{m_2 , n_2}$.
Let $G^i_j(t)$ be the generators of $E_2(K_i)$.
From Lemma~\ref{Pl}, 
for each $i=1,2$, $E_1(K^{m_i,n_i})$ is generated by

\begin{equation*}
\left\{
\begin{split}
&\Delta_{K_i}(t^{\beta_i})\left\{
(1-t^{|m_i|}) , 
(1-t^{\beta_i}) , 
\frac{1-t^{|m_i| \beta_i}}{1-t^{\beta_i}} , 
\frac{1-t^{|m_i| \beta_i}}{1-t^{|m_i|}} , 
\right\} , \\
&G^i_j(t^{\beta_i})(1-t^{|m_i|})\left\{
(1-t^{|m_i|}) , 
(1-t^{\beta_i}) , 
\frac{1-t^{|m_i| \beta_i}}{1-t^{\beta_i}} , 
\frac{1-t^{|m_i| \beta_i}}{1-t^{|m_i|}}
\right\} ,\\
&(1 - t^{|m_i|})^{l_i-1}\left\{
(1-t^{|m_i|}) , 
(1-t^{\beta_i}) , 
\frac{1-t^{|m_i| \beta_i}}{1-t^{\beta_i}} , 
\frac{1-t^{|m_i| \beta_i}}{1-t^{|m_i|}}
\right\} .
\end{split}
\right.
\end{equation*}
Since the ideals $E_1(K_1^{m_1 , n_1})$ and $E_1(K_2^{m_2 , n_2})$ coincide,
each generator of $E_1(K_1^{m_1 , n_1})$ is a linear combination of generators of 
$E_1(K_2^{m_2 , n_2})$ over $\mathbb{Z}[t , t^{-1}]$. 
Thus, for instance, we have
\begin{equation*}
\begin{split}
\Delta_{K _2}(t^{\beta_2})(1-t^{m_2}) =&   
\Delta_{K_1}(t^{\beta_1}) \times \\
&\left\{
P_1(t) (1-t^{|m_1|}) +
P_2(t) (1-t^{\beta_1}) +
P_3(t) \frac{1-t^{|m_1| \beta_1}}{1-t^{\beta_1}} + 
P_4(t) \frac{1-t^{|m_1| \beta_1}}{1-t^{|m_1|}}
\right\} \\
&+\sum_j G^1_j(t^{\beta_1})(1-t^{|m_1|}) \times
\biggl\{P^j_5(t) (1-t^{|m_1|}) + 
P^j_6(t) (1-t^{\beta_1}) + \\
&\left. P^j_7(t)  \frac{1-t^{|m_1| \beta_1}}{1-t^{\beta_1}}
+P^j_8(t)  \frac{1-t^{|m_1| \beta_1}}{1-t^{|m_1|}}\right\} 
 + (1 - t^{|m_1|})^{l_1-1}  \times \\
&\hspace{-10pt}\left\{
P_{9}(t) (1-t^{|m_1|}) +
P_{10}(t) (1-t^{\beta_1}) +
P_{11}(t) \frac{1-t^{|m_1| \beta_1}}{1-t^{\beta_1}} + 
P_{12}(t) \frac{1-t^{|m_1| \beta_1}}{1-t^{|m_1|}}
\right\},
\end{split}
\end{equation*}
where $P_k(t) , P^j_k(t) \in \mathbb{Z}[t , t^{-1}]$ are Laurent polynomials.
Since $m_1$ and $\beta_1$ are relatively prime, $\beta_1$ is odd. 
Substituting $-1$ for the above equation's $t$, we have
 $$
\Delta_{K _1}(-1)
( 2P_2(-1) + \beta_1P_4(-1)) = 
\Delta_{K _2}((-1)^{\beta_2})(1-(-1)^{|m_2|}).
$$
If $m_2$ is even then $2P_2(-1) + \beta_1P_4(-1) = 0$ since $\Delta_K(-1) \neq 0$ for any 1-knot $K$.
If $m_2$ is odd then 
$$
\cfrac{\Delta_{K _2}(1)}{\Delta_{K _1}(-1)} =
\cfrac{1}{\Delta_{K _1}(-1)} = 
\cfrac{2P_2(-1) + \beta_1P_4(-1)}{2} \in \cfrac{\mathbb{Z}}{2}
$$
since $\beta_2$ can be chosen to be even and $\Delta_K(1)=1$ for any 1-knot $K$.

The same arguments for other generators  
$
\Delta_{K _2}(t^{\beta_2})\left\{
(1-t^{\beta_2}) , 
\cfrac{1-t^{|m_2|\beta_2}}{1-t^{\beta_2}} , 
\cfrac{1-t^{|m_2|\beta_2}}{1-t^{|m_2|}}
\right\} 
$
of $E_1(K_2^{m_2 , n_2})$ 
lead the following table:
\begin{center}
\begin{tabular}{c|c|c|c|c}
The generators of $E_1(K_2^{m_2 , n_2})$ &$1-t^{|m_2|}$ & $1-t^{\beta_2}$ &$\cfrac{1-t^{|m_2|\beta_2}}{1-t^{\beta_2}}$ & $\cfrac{1-t^{|m_2|\beta_2}}{1-t^{|m_2|}}$\\
\hline  
$(m_2 , \beta_1 , \beta_2) = (e,o,o)$ & P & $\mathbb{Z}/{2}$ & P & $\mathbb{Z}/{\beta_2}$\\ 
\hline
$(m_2 , \beta_1 , \beta_2) = (o,o,e)$ & $\mathbb{Z}/{2}$ & P & $\mathbb{Z}/{|m_2|}$ & P
\end{tabular}
\end{center}

The second column explains the case of the generator $\Delta_{K _2}(t^{\beta_2})(1-t^{|m_2|})$, which we have seen above.
If $(m_2 , \beta_1 , \beta_2)=(e,o,o)$, where $e$ and $o$ stands for even and odd, respectively,
then we have $2P_2(-1)+\beta_1P_4(-1)=0$, which is represented ``P" in the table.
Note that we cannot get any information of $\Delta_{K_1}(t)$ and $\Delta_{K_2}(t)$ from the information ``P".
If $(m_2 , \beta_1 , \beta_2)=(o,o,e)$ then $\frac{\Delta_{K _2}((-1)^{\beta_2})}{\Delta_{K _1}(-1)} \in \mathbb{Z}/2$, which is represented by ``$\mathbb{Z}/2$".
The $3$rd, $4$th, and $5$th columns are filled by the same way for the other generators 
$\Delta_{K_2}(t^{\beta_2})(1-t^{\beta_2})$, 
$\Delta_{K_2}(t^{\beta_2})\frac{1-t^{|m_2| \beta_2}}{1-t^{\beta_2}}$, 
$\Delta_{K_2}(t^{\beta_2})\frac{1-t^{|m_2| \beta_2}}{1-t^{|m_2|}}$, respectively.

In the case where $m_2$ is odd, from this table, we have $\frac{\Delta_{K _2}((-1)^{\beta_2})}{\Delta_{K _1}(-1)}\in \mathbb{Z}/2 \cap \mathbb{Z}/{|m_2|}$.
We may choose $\beta_2$ to be even, then $|\Delta_{K_1}(-1)|=1$ holds.

In the case where $m_2$ is even, since $\beta_2$ is odd, by the same argument we have $\frac{\Delta_{K_2}(-1)}{\Delta_{K_1}(-1)} \in \mathbb{Z}$. 
By applying the same argument with exchanging $K_1$ and $K_2$, we have $\frac{\Delta_{K _1}(-1)}{\Delta_{K _2}(-1)}\in \mathbb{Z}$.
Thus $|\Delta_{K_2}(-1)|=|\Delta_{K_1}(-1)|$.
\end{proof}

\end{document}